\documentclass[10pt]{amsart}
\usepackage{url,amsmath,amssymb,mathrsfs,commath,amsthm}
\usepackage{enumitem}
\usepackage{makecell}

\newtheorem{theorem}{Theorem}[section]
\newtheorem{lemma}[theorem]{Lemma}
\newtheorem{proposition}[theorem]{Proposition}

\theoremstyle{definition}

\newtheorem*{acknowledgments}{Acknowledgments}

\newcommand{\Tr}{\mathrm{Tr}}
\newcommand{\Gal}{\mathrm{Gal}}
\newcommand{\modulo}[1]{\textup{ (mod }#1)}
\newcommand{\bF}{\mathbb F}
\newcommand{\bN}{\mathbb N}
\newcommand{\bZ}{\mathbb Z}
\newcommand{\bQ}{\mathbb Q}
\newcommand{\gen}[1]{\langle#1\rangle}

\begin{document}
\title[Non-Abelian Extensions]{Non-Abelian extensions of degree $p^3$ and $p^4$ in characteristic $p>2$}
\author{Grant Moles}
\date{\today}

\bibliographystyle{amsalpha}
\begin{abstract}
  This paper describes in terms of Artin-Schreier equations field extensions whose Galois group is isomorphic to any of the four non-cyclic groups of order $p^3$ or the ten non-Abelian groups of order $p^4$, $p$ an odd prime, over a field of characteristic $p$.
\end{abstract}

\maketitle

\newpage

\section{Introduction}

Let $K$ be a field of characteristic $p$, a prime, and $G$ be a $p$-group. The main theorem of \cite{witt} tells us under what circumstances there will exist a field extension $L$ of $K$ such that $L/K$ is Galois with Galois group $G$. Furthermore, in \cite{saltman}, it was shown that such a field extension can be described using polynomials of a particular form, as we will describe shortly (in fact, Saltman showed this more generally for ring extensions). The purpose of this paper is to serve as a natural extension of these important results. For field extensions of this form of degree $p^3$ or $p^4$, we will construct the polynomials based on the presentation of the Galois group $G$.

Let $K_4/K_0$ be a non-Abelian Galois extension of degree $p^4$, where $K_0$ is a field of characteristic $p$, an odd prime. According to Burnside \cite{burnside}, there are ten non-Abelian groups of order $p^4$. Nine of these groups can be expressed as: 
\begin{alignat*}{1}
G_{\vec{a},\vec{b}}=\langle \sigma_1&,\sigma_2,\sigma_3,\sigma_4:\sigma_1^p=\sigma_3^{a_1}\sigma_4^{b_1},\sigma_2^p=\sigma_4^{b_2},\sigma_3^p=\sigma_4^p=1,\sigma_4\in Z(G_{\vec{a},\vec{b}}), \\&\sigma_3\sigma_2=\sigma_2\sigma_3,
\sigma_3\sigma_1=\sigma_1\sigma_3\sigma_4^{b_3},\sigma_2\sigma_1=\sigma_1\sigma_2\sigma_3^{a_0}\sigma_4^{b_0}\rangle,
\end{alignat*}
where $a_i\in\{0,1\}$, $b_j\in\{0,1\}$ for $j\neq2$, and $b_2\in\{0,1,\alpha\}$, where $\alpha$ is a quadratic non-residue mod $p$. The tenth group is expressed as $$H=\langle \sigma_1,\sigma_2,\sigma_3,\sigma_4: \sigma_3=\sigma_1^p,\sigma_4=\sigma_1^{p^2},\sigma_1^{p^3}=\sigma_2^p=1, \sigma_2\sigma_1=\sigma_1^{1+p^2}\sigma_2\rangle.$$
In any of these groups, we can construct the composition series $$\{1\}\trianglelefteq\gen{\sigma_4}\trianglelefteq\gen{\sigma_3,\sigma_4}\trianglelefteq\gen{\sigma_2,\sigma_3,\sigma_4}\trianglelefteq\gen{\sigma_1,\sigma_2,\sigma_3,\sigma_4},$$ with $\gen{\sigma_1,\sigma_2,\sigma_3,\sigma_4}$ being the entire group. The reader will note that this is a central series in any of the given groups, so we can define intermediate fields $K_3,K_2,K_1$ such that $K_3$ is the fixed field of $\langle\sigma_4\rangle$, $K_2$ is the fixed field of $\langle\sigma_3,\sigma_4\rangle$, and $K_1$ is the fixed field of $\langle\sigma_2,\sigma_3,\sigma_4\rangle$. In general, we have $\Gal(K_i/K_{i-1})=\langle \sigma_i|_{K_i}\rangle$, the subgroup of automorphisms of $K_i$ generated by $\sigma_i$ restricted to $K_i$. When the Galois group of $K_4/K_0$ is $G_{\vec{a},\vec{b}}$ as above, we have that $K_3/K_0$ is a non-cyclic Galois extension of degree $p^3$ with Galois group $$G_{\vec{a}}=\gen{\sigma_1,\sigma_2,\sigma_3:\sigma_1^p=\sigma_3^{a_1},\sigma_2^p=\sigma_3^p=1,\sigma_3\in Z(G_{\vec{a}}),\sigma_2\sigma_1=\sigma_1\sigma_2\sigma_3^{a_0}},$$ with $\vec{a}$ the same vector as in the representation of $G_{\vec{a},\vec{b}}$. In fact, this form also holds when the Galois group is $H$, with $a_0=0$, $a_1=1$. Now since each extension $K_i/K_{i-1}$ is cyclic of degree $p$ with Galois group $\gen{\sigma_i|_{K_i}}$, there is some $x_i\in K_i$ such that $K_i=K_{i-1}(x_i)$ which satisfies an Artin-Schreier polynomial with constant coefficient in $K_{i-1}$ such that $(\sigma_i-1)x_i=1$. Note that since $\sigma_i$ fixes $K_j$ for $0\leq j<i\leq 4$,  $(\sigma_i-1)x_j=0$ for $1\leq j<i\leq 4$. Furthermore, we have that $\Gal(K_2/K_0)=\gen{\sigma_1|_{K_2},\sigma_2|_{K_2}}\cong C_p^2$, so we may choose $x_2\in K_2$ to be such that $(\sigma_1-1)x_2=0$. Then we arrive at the main theorem of this paper:

\begin{theorem}
Let $K_4/K_0$ be a non-Abelian Galois extension of degree $p^4$ with $\sigma_i$, $K_i$, and $x_i$ for $1\leq i\leq 4$ as defined above. Then $x_1$ and $x_2$ can be chosen so that $\wp(x_1)\in K_0$ and $\wp(x_2)\in K_0$, where $\wp$ is the Weierstrass $\wp$ function, defined by $\wp(x)=x^p-x$. Denote $\wp(x_1)=\beta_1$ and $\wp(x_2)=\beta_2$. If $\Gal(K_4/K_0)\cong G_{\vec{a},\vec{b}}$ as above, then we can select $x_3$ and $x_4$ such that
\begin{align*}
\wp(x_3)&=a_0\beta_2x_1+a_1D_1(x_1)+\beta_3,\\
\wp(x_4)&=b_0\beta_2x_1+b_1D_1(x_1)+b_2D_1(x_2)+b_3\left(a_0\beta_2\binom{x_1}{2}+\beta_3x_1\right)+\beta_4,
\end{align*} for some $\beta_i\in K_0$, where $D_1\in\bF_p[x]$ the Witt polynomial $$D_1(x)=\frac{x^p+(x^p-x)^p-(x+(x^p-x))^p}{p}.$$ For each group, the specific values of $a_i$ and $b_i$ are shown in Table 1. Otherwise, $\Gal(K_4/K_0)\cong H$ as above, and we can select $x_3$ and $x_4$ such that
\begin{align*}
\wp(x_3)&=D_1(x_1)+\beta_3,\\
\wp(x_4)&=\beta_2 x_1+D_2(x_1,x_3)+\beta_4.
\end{align*}
Again, $\beta_i\in K_0$, $D_1$ is as above, and $D_2(x_1,x_3)\in\bF_p[x_1,x_3]$ is the Witt polynomial
$$D_2(x_1,x_3)=\frac{x_1^{p^2}+\beta_1^{p^2}-(x_1+\beta_1)^{p^2}+p\left(x_3^p+\beta_3^p-(x_3+D_1(x_1)+\beta_3)^p\right)}{p^2}$$
Moreover, the converse holds; that is, any algebraic extension of such a form (the particular form described in more detail later) is a Galois extension with Galois group as presented here.
\end{theorem}
	
	\begin{center}
		\textit{Table 1: Translation to Burnside and James with $\vec{a}$ and $\vec{b}$ Values}
		\begin{tabular}{|c|c|c|c|c|c|c|c|c|c|c|c|}
			\hline
			\hfill Burnside & James & $\sigma_1$ & $\sigma_2$ & $\sigma_3$ & $\sigma_4$ & $a_0$ & $a_1$ & $b_0$ & $b_1$ & $b_2$ & $b_3$\\
			\hline
			(xiv)& $\phi_2(1^4)$ & $S$ & $R$ & $Q$ & $P$ & $0$ & $0$ & $1$ & $0$ & $0$ & $0$\\
			\hline
			(ix) & $\phi_2(211)a$ & $P$ & $R^{-1}$ & $Q$ & $P^p$ & $0$ & $0$ & 1 & 1 & 0 & 0\\
			\hline
			(vii)& $\phi_2(211)b$ & $R$ & $P$ & $Q$ & $P^p$ & 0 & 0 & 0 & 0 & 1 & 1\\
			\hline
			(xv),$p=3$& $\phi_3(1^4)$ & $R$ & $P^{-1}$ & $Q^{-1}$ & $P^p$ & 1 & 0 & 0 & 0 & $\alpha$ & 1\\
			\hline
			(xv),$p>3$& $\phi_3(1^4)$ & $S$ & $R$ & $Q$ & $P$ & 1 & 0 & 0 & 0 & 0 & 1\\
			\hline
			(x)& $\phi_2(211)$ & $P$ & $R$ & $Q^{-1}$ & $P^p$ & 1 & 0 & 0 & 1 & 0 & 0\\
			\hline
			(xi) & \makecell{$\phi_3(211)b_1,p=3$\\$\phi_3(211)a,p>3$} & $P$ & $R$ & $Q^{-1}$ & $P^p$ & 1 & 0 & 0 & 1 & 0 & 1\\
			\hline
			(xii) & \makecell{$\phi_3(211)b_\nu,p=3$\\$\phi_3(211)b_1,p>3$}& $P$ & $R$ & $Q^{-1}$ & $P^p$ & 1 & 0 & 0 & 1 & 1 & 1\\
			\hline
			(xiii) & \makecell{$\phi_3(211)a,p=3$\\$\phi_3(211)b_\nu,p>3$} & $P$ & $R$ & $Q^{-1}$ & $P^p$ & 1 & 0 & 0 & 1 & $\alpha$ & 1\\
			\hline
			(viii) & $\phi_2(22)$ & $P$ & $Q^{-1}$ & $P^p$ & $Q^{-p}$ & 1 & 1 & 0 & 0 & 1 & 0\\
			\hline
			(vi) & $\phi_2(31)$ & $P$ & $Q^{-1}$ & $P^p$ & $P^{p^2}$ & - & - & - & - & - & - \\
			\hline
		\end{tabular}
	\end{center}

	In Table 1, $\alpha$ refers to a quadratic non-residue mod $p$ (recall that when $p=3$, the only quadratic non-residue is $\alpha=-1$). The value for $\wp(x_3)$ is determined below in Proposition 2.2, and the value for $\wp(x_4)$ in Propositions 3.1 and 3.3. It is worth noting that not every combination of vectors $\vec{a}$ and $\vec{b}$ is considered here, though every non-Abelian group is represented. Notably, some other combinations of $\vec{a}$ and $\vec{b}$ will yield Abelian groups, and some will yield equivalent representations of the non-Abelian groups. For example, when both $\vec{a}$ and $\vec{b}$ are zero vectors, $G_{\vec{a},\vec{b}}\cong C_p^4$. Since the combinations presented cover all non-Abelian groups, only these combinations will be considered in this paper. Also shown in the table are the notations for the groups used by Burnside \cite{burnside} and James \cite{james}, as well as the translation from the group elements $P,Q,R,S$ used by Burnside to the $\sigma_i$ defined above. The order has been changed to place groups with similar corresponding Artin-Schreier equations near each other. For more information on the groups presented here, as well as other groups of order $p^n$ for $n\leq 6$, see \cite{james}.
	
	At this point, it should be noted that although the work presented here follows naturally from Saltman's work, the notation used will be that presented above, rather than that used in \cite{saltman}. In particular, the elements $x_i$ and $\beta_i$ in this paper are referred to as $\theta_i$ and $r_i$, respectively, by Saltman. In his notation, we are constructing the particular polynomials $s_i$ corresponding to each choice of Galois group $G$. It it worth noting that, in Theorem 1.1 above, we have $\beta_i=x_i^p-x_i$ for $i=1,2$, and so $\beta_3$ is actually equal to $\wp(x_3)$ plus a polynomial in $\bF_p[x_1,x_2]$ (and similarly for $\beta_4$), as we would expect.
	
	The work presented here has potential applications in group theory, namely the study of $p$-groups and related structures, as well as Galois theory. Perhaps most obviously, the methods used to determine group actions and Artin-Schreier equations throughout this paper could be extended to extensions defined by other Galois groups. Most immediately, this could be applied to Galois extensions of degree $p^n$ for $n\geq 5$. However, this could also extend to other subgroups of the Nottingham group, especially those which are non-Abelian (see \cite{camina}, \cite{bct}). The results in this paper could also be used to approach the local lifting problem for the groups outlined above. For more information on this problem, as well as a characterization of extensions with Galois group $D_4$ in characteristic 2, see \cite{weaver}. Finally, results such as those found here are of particular interest and utility when considering ramification indices of local field extensions with these Galois groups (see \cite{elder1}, \cite{elder2}, \cite{elder3}).

\section{Extensions of degree $p^3$}

Before presenting the arguments for extensions of degree $p^4$, it will help to consider those of degree $p^3$. As discussed, we will only consider non-cyclic groups of order $p^3$, of which there are four: $C_p^3$, $C_{p^2}\times C_p$, the Heisenberg group $H(p^3)$, and the Metacyclic group $M(p^3)$. These can all be expressed as
\begin{equation}
	G_{\vec{a}}=\langle \sigma_1,\sigma_2,\sigma_3: \sigma_1^p=\sigma_3^{a_1}, \sigma_2^p=\sigma_3^p=1, \sigma_3\in Z(G_{\vec{a}}), \sigma_2\sigma_1=\sigma_1\sigma_2\sigma_3^{a_0}\rangle,
\end{equation}
where $\vec{a}=(a_0,a_1)$ takes on the values shown in Table 2. 
\begin{center}
	\textit{Table 2: Values of $\vec{a}$}\\
	\begin{tabular}{|c|c|c|}
		\hline
		Group & $a_0$ & $a_1$\\
		\hline
		$C_p^3$ & 0 & 0\\
		\hline
		$C_{p^2}\times C_p$ & 0 & 1\\
		\hline
		$H(p^3)$ & 1 & 0\\
		\hline
		$M(p^3)$ & 1 & 1\\
		\hline
	\end{tabular}
\end{center}
Although in these groups, $a_0$ and $a_1$ only take on values of 0 or 1, it will later be convenient to consider more general $a_i\in \bF_p$ (any such choice of $a_1$ and $a_2$ will lead to a different representation of one of these same groups). As such, we will allow $a_1$ and $a_2$ to take on values other than 0 or 1 in our results. This representation will lead us to the desired result on the polynomials defining the extension $K_0/K_0$, but first we need a lemma.

\begin{lemma}
	Let $K$ be a field of characteristic $p$ and $L/K$ a Galois extension of degree $p$ with Galois group $G=\gen{\sigma_1}$. Let $x_1\in L$ be such that $L=K(x_1)$, $(\sigma_1-1)x_1=1$ and $\wp(x_1)\in K$ (call this element $\wp(x_1)=\beta_1$). Let $C_1$ be the polynomial $$C_1(x_1)=\frac{x_1^p+1-(x_1+1)^p}{p}=-\sum_{i=1}^{p-1}\frac{\binom{p}{i}}{p}x_1^i\in\bF_p[x_1].$$ Then $\Tr_{L/K}(C_1(x_1))=(1+\sigma_1+\dots+\sigma_1^{p-1})C_1(x_1)=1$ and $\wp(C_1(x_1))=(\sigma_1-1)D_1(x_1)$, where $$D_1(x_1)=\frac{x_1^p+\beta_1^p-(x_1+\beta_1)^p}{p}=-\sum_{i=1}^{p-1}\frac{\binom{p}{i}}{p}x_1^i\beta_1^{p-i}\in\bF_p[x_1].$$
\end{lemma}

\begin{proof}
	First, note that $(\sigma_1-1)x_1=1$, i.e. $\sigma_1(x_1)=x_1+1$, implies that $\sigma_1^k(x_1)=x_1+k$ for each $k\in\bN$. Then using the fact that $\sigma_1$ is an automorphism of $L$ which fixes $K$ (and thus $\bF_p$), we have $$\Tr_{L/K}(C_1(x_1))=(1+\sigma_1+\dots+\sigma_1^{p-1})C_1(x_1)=\sum_{j=0}^{p-1}\left(-\sum_{i=1}^{p-1}\frac{\binom{p}{i}}{p}(x_1+j)^i\right).$$
	To get the desired result, it will now be easiest to show a polynomial identity in the ring $\bQ[x]$.
	\begin{align*}
		\sum_{j=0}^{p-1}\left(-\sum_{i=1}^{p-1}\frac{\binom{p}{i}}{p}(x+j)^i\right)&=\sum_{j=0}^{p-1}\frac{(x+j)^p+1-(x+j+1)^p}{p}\\
		&=\frac{x^p+1-(x+1)^p}{p}+\dots+\frac{(x+p-1)^p+1-(x+p)^p}{p}\\
		&=\frac{x^p-(x+p)^p+p}{p}\\
		&=1-\sum_{i=0}^{p-1}\frac{\binom{p}{i}}{p}x^ip^{p-i}.
	\end{align*}
	Note that since the two ends of this string of identities lie in $\bZ[x]$, this identity will still hold in this context. Furthermore, we can reduce modulo $p$ and plug in $x=x_1$ to return to the context of $\bF[x_1]$:
	$$\Tr_{L/K}(C_1(x_1))=\sum_{j=0}^{p-1}\left(-\sum_{i=1}^{p-1}\frac{\binom{p}{i}}{p}(x_1+j)^i\right)=1.$$
	This shows the desired result for the trace; we now turn our attention to $\wp(C_1(x_1)).$ First, recall that $\wp$ is an $\bF_p$-linear map; then 
	\begin{align*}
		\wp(C_1(x_1))&=C_1(x_1^p)-C_1(x_1)\\&=C_1(x_1+\beta_1)-C_1(x_1)
		\\&=-\sum_{i=1}^{p-1}\frac{\binom{p}{i}}{p}(x_1+\beta_1)^i+\sum_{i=1}^{p-1}\frac{\binom{p}{i}}{p}x_1^i.
	\end{align*}
	As before, simplifying this expression will be easier to do after some intermediate steps carried out in $\bQ[x,y]$:
	\begin{align*}
		-\sum_{i=1}^{p-1}\frac{\binom{p}{i}}{p}(x+y)^i+\sum_{i=1}^{p-1}\frac{\binom{p}{i}}{p}x^i&=\frac{(x+y)^p+1-(x+y+1)^p}{p}-\frac{x^p+1-(x+1)^p}{p}\\&=\frac{(x+y)^p-(x+y+1)^p-x^p+(x+1)^p}{p}\\&=\frac{(x+1)^p+y^p-(x+y+1)^p}{p}-\frac{x^p+y^p-(x+y)^p}{p}\\&=-\sum_{i=1}^{p-1}\frac{\binom{p}{i}}{p}(x+1)^{i}y^{p-i}+\sum_{i=1}^{p-1}\frac{\binom{p}{i}}{p}x^iy^{p-i}.
	\end{align*}
	Again, note that both ends of this string of identities lie in $\bZ[x]$, so the identity will still hold true after reducing modulo $p$ and substituting $x=x_1$ and $y=\beta_1$. Then
	\begin{align*}
		\wp(C_1(x_1))&=-\sum_{i=1}^{p-1}\frac{\binom{p}{i}}{p}(x_1+\beta_1)^i+\sum_{i=1}^{p-1}\frac{\binom{p}{i}}{p}x_1^i\\
		&=-\sum_{i=1}^{p-1}\frac{\binom{p}{i}}{p}(x_1+1)^{i}\beta_1^{p-i}+\sum_{i=1}^{p-1}\frac{\binom{p}{i}}{p}x_1^i\beta_1^{p-i}\\
		&=(\sigma_1-1)D_1(x_1).
	\end{align*}
\end{proof}

\begin{proposition}
	Let $K_3/K_0$ be a Galois extension of degree $p^3$ with Galois group $\Gal(K_3/K_0)\cong G_{\vec{a}}$ as in $\textup{(1)}$,
	with $a_0,a_1\in \bF_p$ and $K_2/K_0$ as defined in the introduction. That is, $K_2=K_0(x_1,x_2)$ is the fixed field of $\gen{\sigma_3}\leq G_{\vec{a}}$, $K_1=K_0(x_1)$ is the fixed field of $\gen{\sigma_2,\sigma_3}$, and $(\sigma_i-1)x_i=1$ for $1\leq i\leq 3$, with $\wp(x_1)=\beta_1$ and $\wp(x_2)=\beta_2$ for some $\beta_1,\beta_2\in K_0$. Then we can select $x_3$ such that $K_3=K_0(x_1,x_2,x_3)$, with $x_3^p-x_3=a_0 \beta_2 x_1+a_1 D_1(x_1)+\beta_3$ for some $\beta_3\in K_0$, where $D_1$ is the Witt polynomial from Theorem 1.1. 
	
	Moreover, the converse holds; that is, if $K_3/K_0$ is an algebraic extension of degree $p^3$ with $K_3=K_0(x_1,x_2,x_3)$ and $\wp(x_i)$ as above for $i=1,2,3$, then $K_3/K_0$ is a Galois extension with Galois group $G_{\vec{a}}$.
\end{proposition}
\begin{proof}
	First, recall that since $\Gal(K_2/K_0)$ is a $C_p^2$ group, we may choose $x_2\in K_2$ such that $(\sigma_1-1)x_2=0$. Also note that $\Gal(K_3/K_1)=\langle\sigma_2,\sigma_3\rangle\cong C_p^2$. Then when considering the extension $K_3/K_1$, we can choose $x_3\in K_3$ such that $(\sigma_2-1)x_3=0$. Then the only unknown part of the group action is $\sigma_1$ on $x_3$. Let $A=(\sigma_1-1)x_3\in K_3$. Since $\sigma_3\in Z(G_{\vec{a}})$, $(\sigma_3-1)A=(\sigma_1-1)(\sigma_3-1)x_3=(\sigma_1-1)1=0$. Thus, $A\in K_2$. Now using a known identity from the group definition, we can derive the following identity:
	\begin{alignat*}{1}
		x_3+\sigma_2(A)&=\sigma_2(x_3+A)\\&=\sigma_2\sigma_1(x_3)\\
		&=\sigma_1\sigma_2\sigma_3^{a_0}(x_3)\\
		&=\sigma_1\sigma_2(x_3+a_0)\\
		&=\sigma_1(x_3+a_0)\\
		&=x_3+A+a_0.
	\end{alignat*}
	Then $(\sigma_2-1)A=a_0$. Therefore, $(\sigma_2-1)(A-a_0 x_2)=0$, so $A-a_0 x_2\in K_1$.
	
	Now $\Tr_{K_1/K_0}(A-a_0 x_2)=(\sigma_1^p-1)x_3=(\sigma_3^{a_1}-1)x_3=a_1$. Then by the lemma, $\Tr_{K_1/K_0}(A-a_0 x_2-a_1 C_1(x_1))=0$. Using the additive version of Hilbert's Theorem 90, we then know that there is some $k_1\in K_1$ such that $(\sigma_1-1)k_1=A-a_0 x_2-a_1 C_1(x_1)$. Then $(\sigma_1-1)(x_3-k_1)=a_0 x_2+a_1 C_1(x_1)$. We can then replace $x_3$ with $x_3-k_1$, which maintains all prior relationships and establishes
	\begin{equation}
	(\sigma_1-1)x_3=a_0 x_2+a_1 C_1(x_1).
	\end{equation} 
	It will also be helpful to use (2) to establish another identity. For this, note that in characteristic $p$, $1+\sigma_1+\dots+\sigma_1^{p-1}=(\sigma_1-1)^{p-1}$.
	\begin{equation}
	\begin{alignedat}{1}
	(1+\sigma_1+\dots+\sigma_1^{p-1})x_3&=(\sigma_1-1)^{p-2}(a_0x_2+a_1C_1(x_1))\\
	&=a_1(\sigma_1-1)^{p-2}C_1(x_1).
	\end{alignedat}
	\end{equation}
	
	Now consider $B=\wp(x_3)\in K_2$. Recall that $\wp$ is an $\bF_p$-linear map and an additive homomorphism which commutes with the elements of $G_{\vec{a}}$; then $(\sigma_2-1)B=\wp((\sigma_2-1)x_3)=0$, so $B\in K_1$. Now note that:
	\begin{alignat*}{1}
		(\sigma_1-1)B&=\wp((\sigma_1-1)x_3)\\
		&=\wp(a_0 x_2+a_1 C_1(x_1))\\
		&=a_0 \beta_2 + a_1(\sigma_1-1)D_1(x_1)\\
		&=(\sigma_1-1)(a_0 \beta_2 x_1+a_1 D_1(x_1)),
	\end{alignat*}
	so $(\sigma_1-1)(B-a_0 \beta_2x_1-a_1 D_1(x_1))=0$. For ease of notation, define $A_1:=\beta_2 x_1$. Then $B-a_0 A_1 - a_1 D_1(x_1)=\beta_3$ for some $\beta_3\in K_0$. Therefore, 
	\begin{equation*}
		B=\wp(x_3)=a_0 A_1+a_1 D_1(x_1)+\beta_3.
	\end{equation*}
	
	Now for the converse, let $K_3/K_0$ be an algebraic extension of degree $p^3$ and $K_3=K_0(x_1,x_2,x_3)$, with $\wp(x_3)=a_0A_1+a_1D_1(x_1)+\beta_3$, $\wp(x_2)=\beta_2$, and $\wp(x_1)=\beta_1$, where $\beta_i\in K_0$, $A_1=\beta_2x_1$, and $a_i\in\{0,1\}$. Define $G=\Gal(K_3/K_0)$ (the group of automorphisms of $K_3$ which fix $K_0$ pointwise), $K_1=K_0(x_1)$ and $K_2=K_1(x_2)$ as used previously, so that $K_3=K_2(x_3)$. Then each $K_{i}/K_{i-1}$ is an extension of degree $p$ for $1\leq i\leq 3$, and the Artin-Schreier equations given are exactly the minimal polynomials for each $x_i$ over $K_{i-1}$. Immediately, we know that $K_2/K_0$ is a Galois extension of order $p^2$ (and in fact, is a $C_p^2$ extension), since $x_1$ and $x_2$ each satisfy an Artin-Schreier polynomial over $K_0$. We also have that $K_3$ is trivially Galois over both $K_2$ and $K_1(x_3)$. Then we need to show that $K_3$ is Galois over $K_0$.
	
	Let $\sigma_3\in \Gal(K_3/K_2)\subseteq G$ be such that $\sigma_3(x_3)=x_3+1$. Then $\sigma_3(x_1)=x_1$, and $\sigma_3(x_2)=x_2$. Also let $\sigma_2\in\Gal(K_3/K_1(x_3))\subseteq G$ such that $\sigma_2(x_2)=x_2+1$. Then $\sigma_2(x_3)=x_3$, and $\sigma_2(x_1)=x_1$. Now let $\sigma_1$ be an embedding of $K_3$ into its algebraic closure which fixes $K_0$ pointwise and for which $\sigma_1(x_1)\neq x_1$. Then $$\sigma_1(x_1)^p-\sigma_1(x_1)-\beta_1=\sigma_1(x_1^p-x_1-\beta_1)=0,$$ so $\sigma_1(x_1)$ must be another root of $x^p-x-\beta_1$; then $\sigma_1(x_1)=x_1+n$ for some $n\in\bF_p$. Without loss of generality, assume $\sigma_1(x_1)=x_1+1$ (if not, some power of $\sigma_1$ works). Furthermore, we know that $\sigma_1(x_2)=x_2+m$ for some $m\in\bF_p$. Without loss of generality, we can assume that $m=0$ (if not, multiply $\sigma_1$ by a power of $\sigma_2$). Finally, note that
	\begin{alignat*}{1}
		\wp(\sigma_1(x_3))&=\sigma_1(\wp(x_3))\\&=\sigma_1(a_0A_1+a_1D_1(x_1)+\beta_3)\\&=a_0A_1+a_0\beta_2+a_1D_1(x_1+1)+\beta_3.
	\end{alignat*}
	That is, $\sigma_1(x_3)$ is a root of the polynomial $x^p-x-(a_0\beta_2(x_1+1)+a_1D_1(x_1+1)+\beta_3)$, a degree $p$ polynomial over $K_1$. The $p$ roots of this polynomial are of the form $x_3+a_0x_2+a_1C_1(x_1)+n$, where $n\in\bF_p$ (using identities shown above). Then $\sigma_1(x_3)$ must be of this form, and therefore lies in $K_3$. Without loss of generality, we can assume that $n=0$ (if not, multiply $\sigma_1$ by the appropriate power of $\sigma_3$). Now note that $\sigma_1(x_1)=x_1+1$, $\sigma_1(x_2)=x_2$, and $\sigma_1(x_3)=x_3+a_0x_2+a_1C_1(x_1)$ (and $\sigma_1$ fixes $K_0$ and respects addition and multiplication), so $\sigma_1$ is an automorphism of $K_3$ which fixes $K_0$, i.e. $\sigma_1\in G$. However, since $\sigma_1$ does not fix $x_1$, $\sigma_1\notin\gen{\sigma_2,\sigma_3}$. Then since $\abs{\gen{\sigma_2,\sigma_3}}=p^2$ and $\abs{G}$ must divide $p^3$, it follows that $\abs{G}=p^3$ and $G=\gen{\sigma_1,\sigma_2,\sigma_3}$; in particular, this means that $K_3$ is Galois over $K_0$. Furthermore, we have
	\begin{alignat*}{1}
		\sigma_1^p(x_1)&=x_1+p\\&=x_1,\\
		\sigma_1^p(x_2)&=x_2,\\
		\sigma_1^p(x_3)&=x_3+a_0(px_2)+a_1(C_1(x_1)+C_1(x_1+1)+\dots+C_1(x_1+p-1))\\&=x_3+a_1.
	\end{alignat*}
	Therefore, $\sigma_1^p=\sigma_3^{a_1}$. Finally, note that $\sigma_1$ commutes with $\sigma_3$ and $\sigma_2\sigma_1=\sigma_1\sigma_2\sigma_3^{a_0}$. Therefore, we have $$G=G_{\vec{a}}=\gen{\sigma_1,\sigma_2,\sigma_3:\sigma_1^p=\sigma_3^{a_1},\sigma_2^p=\sigma_3^p=1,\sigma_3\in Z(G_{\vec{a}}),\sigma_2\sigma_1=\sigma_1\sigma_2\sigma_3^{a_0}}.$$ This completes the proof of the converse.
		
\end{proof}

\section{Extensions of degree $p^4$}

We can now discuss the Artin-Schreier equations for the non-Abelian degree $p^4$ extensions discussed in the introduction. Note that since $K_3/K_0$ is a degree $p^3$ extension, we already know $\wp(x_3)$ (and more trivially, $\wp(x_1)$ and $\wp(x_2)$). We will first focus on the first nine cases. As discussed in the introduction, these Galois groups are all of the form
\begin{equation}
\begin{alignedat}{1}
	G_{\vec{a},\vec{b}}=\langle \sigma_1&,\sigma_2,\sigma_3,\sigma_4:\sigma_1^p=\sigma_3^{a_1}\sigma_4^{b_1},\sigma_2^p=\sigma_4^{b_2},\sigma_3^p=\sigma_4^p=1,\sigma_4\in Z(G_{\vec{a},\vec{b}}),\\ &\sigma_3\sigma_2=\sigma_2\sigma_3,
	\sigma_3\sigma_1=\sigma_1\sigma_3\sigma_4^{b_3},\sigma_2\sigma_1=\sigma_1\sigma_2\sigma_3^{a_0}\sigma_4^{b_0}\rangle,
\end{alignedat}
\end{equation}
where $\vec{a}=(a_0,a_1)$ and $\vec{b}=(b_0,b_1,b_2,b_3)$ take on the values shown in Table 1 above. Recall that $\alpha$ refers to any quadratic non-residue (mod p), and in the case of $p=3$, the only non-residue is $\alpha=2\equiv -1$. This representation leads us to our main proposition:

\begin{theorem}
	Let $K_4/K_0$ be a Galois extension of degree $p^4$ with Galois group $G_{\vec{a},\vec{b}}$ as shown in equation \textup{(4)} above, where $a_0,a_1,b_0,b_1,b_3\in \lbrace 0,1\rbrace$, $b_2\in \lbrace 0,1,\alpha\rbrace$, $\alpha$ is a quadratic non-residue modulo $p$, and all notation is as above. Assume that $K_3/K_0$, $x_1$, $x_2$, and $x_3$ are as in Proposition 2.2. Then we can select $x_4$ such that $K_4=K_3(x_4)$ and
	\begin{equation*}
		x_4^p-x_4=b_0 \beta_2 x_1+b_1 D_1(x_1)+b_2 D_1(x_2) + b_3 \left(a_0 \beta_2\binom{x_1}{2}+\beta_3 x_1\right)+\beta_4,
	\end{equation*}
	where $\beta_4\in K_0$ and $D_1$ is the Witt polynomial from Theorem 1.1.
	
	Moreover, the converse holds; that is, if $K_4/K_0$ is an algebraic extension of degree $p^4$ with $K_4=K_0(x_1,x_2,x_3,x_4)$ and $\wp(x_i)$ as above for $i=1,2,3,4$, then $K_4/K_0$ is a Galois extension with Galois group $G_{\vec{a},\vec{b}}$.
\end{theorem}
  
\begin{proof}
	First, since $\Gal(K_3/K_0)$ is as in Proposition 2.2, we have $\wp(x_3)=a_0 A_1+a_1 D_1(x_1)+\beta_3$, with $(\sigma_2-1)x_3=0$ and $(\sigma_1-1)x_3=a_0 x_2+a_1 C_1(x_1)$. Note that $K_4/K_1$ is another extension of degree $p^3$. From the expression of $G_{\vec{a},\vec{b}}$, we can see that the Galois group for $K_4/K_1$ is as shown in equation (1) in Section 2, with $a_0=0$ and $a_1=b_2$ (i.e. the Galois group is either $C_p^3$ or $C_{p^2}\times C_p$, both of which are Abelian). Then by Proposition 2.2, we can choose $x_4\in K_4$ such that $(\sigma_3-1)x_4=0$ and $(\sigma_2-1)x_4=b_2 C_1(x_2)$.
	
	Now all that remains is to determine the action of $\sigma_1$ on $x_4$. To that end, let $A=(\sigma_1-1)x_4\in K_4$. Since $\sigma_4\in Z(G_{\vec{a},\vec{b}})$, $(\sigma_4-1)A=(\sigma_1-1)(\sigma_4-1)x_4=(\sigma_1-1)1=0$. Then $A\in K_3$. Now using a known identity from the group definition, we can derive the following identity:
	\begin{alignat*}{1}
		x_4+\sigma_3(A)&=\sigma_3(x_4+A)\\&=\sigma_3\sigma_1(x_4)\\
		&=\sigma_1\sigma_3\sigma_4^{b_3}(x_4)\\
		&=\sigma_1\sigma_3(x_4+b_3)\\
		&=\sigma_1(x_4+b_3)\\
		&=x_4+A+b_3
	\end{alignat*}
	Then $(\sigma_3-1)A=b_3$. Therefore, $(\sigma_3-1)(A-b_3 x_3)=0$, so $A-b_3 x_3\in K_2$. Now using another identity from the group definition:
	\begin{alignat*}{1}
		x_4+b_2C_1(x_2)+\sigma_2(A)&=\sigma_2(x_4+A)\\&=\sigma_2\sigma_1(x_4)\\
		&=\sigma_1\sigma_2\sigma_3^{a_0}\sigma_4^{b_0}(x_4)\\
		&=\sigma_1\sigma_2\sigma_3^{a_0}(x_4+b_0)\\
		&=\sigma_1\sigma_2(x_4+b_0)\\
		&=\sigma_1(x_4+b_2C_1(x_2)+b_0)\\
		&=x_4+A+b_2C_1(x_2)+b_0
	\end{alignat*}
	Then $(\sigma_2-1)A=b_0$. Since $x_3$ is fixed by $\sigma_2$, $(\sigma_2-1)(A-b_0 x_2-b_3 x_3)=0$, so $A-b_0 x_2-b_3 x_3\in K_1$.
	
	Now using the action of $\sigma_1$ on each element as well as equation (3) in Proposition 2.2, $\Tr_{K_1/K_0}(A-b_0 x_2-b_3 x_3)=(\sigma_1^p-1)x_4-p(b_0x_2)-a_1b_3(\sigma_1-1)^{p-2}C_1(x_1)=(\sigma_3^{a_1}\sigma_4^{b_1}-1)x_4=b_1$. It should be noted that this holds because either $a_1$ or $b_3$ must be 0, as seen in Table 3 (the fact that one of these constants must be 0 can be shown more generally from the group definition, though this is left to the reader). Then using the identity $\Tr_{K_1/K_0}(C_1(x_1))=1$, shown in Lemma 2.1, $\Tr_{K_1/K_0}(A-b_0 x_2-b_1C_1(x_1)-b_3 x_3)=0$. Using the additive version of Hilbert's Theorem 90, we then know that there is some $k_1\in K_1$ such that $(\sigma_1-1)k_1=A-b_0 x_2-b_1C_1(x_1)-b_3 x_3$. Then $(\sigma_1-1)(x_4-k_1)=b_0 x_2+b_1C_1(x_1)+b_3 x_3$. We can now replace $x_4$ with $x_4-k_1$, which maintains all prior relationships and establishes the identity 
	\begin{equation}
	(\sigma_1-1)x_4=b_0 x_2+b_1C_1(x_1)+b_3 x_3.
	\end{equation}
		
	Now consider $B=\wp(x_4)\in K_3$. Recall that $\wp$ is an $\bF_p$-linear map and an additive homomorphism which commutes with the elements of $G_{\vec{a},\vec{b}}$. Then $(\sigma_3-1)B=\wp((\sigma_3-1)x_4)=0$. Then $B\in K_2$. Now note that
	\begin{alignat*}{1}
		(\sigma_2-1)B&=\wp((\sigma_2-1)x_4)\\
		&=\wp(b_2 C_1(x_2))\\
		&=(\sigma_2-1)(b_2 D_1(x_2)),
	\end{alignat*}
	 so $(\sigma_2-1)(B-b_2 D_1(x_2))=0$. Then $B-b_2 D_1(x_2)\in K_1$. Now note the following, recalling that $a_0b_3=0$:
	 \begin{alignat*}{1}
	 	(\sigma_1-1)(B-b_2 D_1(x_2))&=\wp((\sigma_1-1)x_4)\\
	 	&=\wp(b_0 x_2+b_1C_1(x_1)+b_3 x_3)\\
	 	&=b_0 \beta_2 +b_1(\sigma_1-1)D_1(x_1)+b_3(a_0 \beta_2 x_1 +a_1 D_1(x_1) +\beta_3)\\
	 	&=(\sigma_1-1)\left(b_0 \beta_2 x_1+b_1 D_1(x_1)+b_3\left(a_0\beta_2 \binom{x_1}{2}+\beta_3 x_1\right)\right).
	 \end{alignat*}
	  For ease of notation, define $A_1=\beta_2 x_1$, $A_2=\beta_2\binom{x_1}{2}$, and $A_3=\beta_3 x_1$. Then $(\sigma_1-1)(B-b_0 A_1 -b_1 D_1(x_1)-b_2 D_1(x_2)-b_3(a_0 A_2+A_3))=0$. Therefore, $B-b_0 A_1 -b_1 D_1(x_1)-b_2 D_1(x_2)-b_3(a_0 A_2+A_3)=\beta_4$ for some $\beta_4\in K_0$. Then
	  \begin{equation*}
		  B=b_0 A_1 +b_1 D_1(x_1)+b_2 D_1(x_2)+b_3(a_0 A_2+A_3)+\beta_4. 
	  \end{equation*}
	  
	  Now for the converse, let $K_4/K_0$ be an algebraic extension of degree $p^4$ and $K_4=K_0(x_1,x_2,x_3,x_4)$, with $\wp(x_4)=b_0A_1+b_1D_1(x_1)+b_2D_1(x_2)+b_3(a_0A_2+A_3)+\beta_4$, $\wp(x_3)=a_0A_1+a_1D_1(x_1)+\beta_3$, $\wp(x_2)=\beta_2$, $\wp(x_1)=\beta_1$, where $\beta_i\in K_0$, $A_1=\beta_2x_1$, $A_2=\beta_2\binom{x_1}{2}$, $A_3=\beta_3x_1$, $a_i\in\{0,1\}$, $b_i\in\{0,1\}$ for $i\neq 2$, and $b_2\in\{0,1,\alpha\}$, where $\alpha$ is a quadratic non-residue modulo $p$. Define $G=\Gal(K_4/K_0)$ (the group of automorphisms of $K_4$ which fix $K_0$ pointwise), $K_1=K_0(x_1)$, $K_2=K_1(x_2)$, and $K_3=K_2(x_3)$ as used previously, so that $K_4=K_3(x_4)$. Then each $K_{i}/K_{i-1}$ is an extension of degree $p$ for $1\leq i\leq 4$, and the Artin-Schreier equations given are exactly the minimal polynomials for each $x_i$ over $K_{i-1}$. Immediately from the proof of Proposition 2.2, we know that $K_3/K_0$ and $K_4/K_1$ are Galois extensions of order $p^3$. We also have that $K_4$ is trivially Galois over both $K_3$ and $K_2(x_4)$. Then we need to show that $K_4$ is Galois over $K_0$.
	  
	  Let $\sigma_4\in \Gal(K_4/K_3)\subset G$ such that $\sigma_4(x_4)=x_4+1$. Then $\sigma_4$ fixes $x_1$, $x_2$, and $x_3$. Also let $\sigma_3\in \Gal(K_4/K_2(x_4))$ such that $\sigma_3(x_3)=x_3+1$. Then $\sigma_3$ fixes $x_1$, $x_2$, and $x_4$. Using the proof of Proposition 2.2, we also have $\sigma_2\in \Gal(K_4/K_1)$ such that $\sigma_2(x_2)=x_2+1$, $\sigma_2(x_3)=x_3$, and $\sigma_2(x_4)=x_3+b_2C_1(x_2)$.  Now let $\sigma_1$ be an embedding of $K_4$ into its algebraic closure which fixes $K_0$ pointwise and for which $\sigma_1(x_1)\neq x_1$. Then $\sigma_1|_{K_3}$ is an element of $\Gal(K_3/K_0)$ which does not fix $x_1$. Then we can again use the proof of Proposition 2.2 to assume that $\sigma_1(x_1)=x_1+1$, $\sigma_1(x_2)=x_2$, and $\sigma_1(x_3)=x_3+a_0x_2+a_1C_1(x_1)$. Now note the following, recalling that $a_0b_3=0$:
	  \begin{alignat*}{1}
	  	\wp(\sigma_1(x_4))=&\sigma_1(\wp(x_4))\\
	  	=&\sigma_1(b_0A_1+b_1D_1(x_1)+b_2D_1(x_2)+b_3(a_0A_2+A_3)+\beta_4)\\
	  	=&b_0(A_1+\beta_2)+b_1D_1(x_1+1)+b_2D_1(x_2)\\&+b_3\left(a_0\beta_2\binom{x_1+1}{2}+A_3+\beta_3\right)+\beta_4
	  \end{alignat*}
	  That is, $\sigma_1(x_4)$ is a root of the polynomial $x^p-x-\wp(\sigma_1(x_4))$, where $\wp(\sigma_1(x_4))$ is the expression derived above. This is a degree $p$ polynomial over $K_2$, and the $p$ distinct elements $x_4+b_0x_2+b_1C_1(x_1)+b_3x_3+n$, with $n\in\bF_p$, are roots. Then these are the only roots, and so $\sigma_1(x_4)$ is of this form for some $n\in\bF_p$ (and thus lies in $K_4$). Without loss of generality, we will assume that $n=0$ (if not, multiply $\sigma_1$ by the appropriate power of $\sigma_4$). Now note that $\sigma_1(x_1)=x_1+1$, $\sigma_1(x_2)=x_2$, $\sigma_1(x_3)=x_3+a_0x_2+a_1C_1(x_1)$, and $\sigma_1(x_4)=x_4+b_0x_2+b_1C_1(x_1)+b_3x_3$ (and $\sigma_1$ fixes $K_0$ and respects addition and multiplication), so $\sigma_1$ is an automorphism of $K_4$ which fixes $K_0$, i.e. $\sigma_1\in G$. However, since $\sigma_1$ does not fix $x_1$, $\sigma_1\notin\gen{\sigma_2,\sigma_3,\sigma_4}=\Gal(K_4/K_1)$. Then since $\abs{G}>\abs{\Gal(K_4/K_1)}=p^3$ and $\abs{G}$ must divide $p^4$, it follows that $\abs{G}=p^4$ and $G=\gen{\sigma_1,\sigma_2,\sigma_3,\sigma_4}$; in particular, this means that $K_4$ is Galois over $K_0$. Furthermore, we have
	  \begin{alignat*}{1}
	  	\sigma_1^p(x_1)=&x_1,\\
	  	\sigma_1^p(x_2)=&x_2,\\
	  	\sigma_1^p(x_3)=&x_3+a_1,\\
	  	\sigma_1^p(x_4)=&x_4+pb_0x_2+b_1(C_1(x_1)+\dots+C_1(x_1+p-1))\\&+a_1b_3\left((p-1)C_1(x_1)+(p-2)C_1(x_1+1)+\dots+C_1(x_1+p-2)\right)\\
	  	=&x_4+b_1.
	  \end{alignat*}
	  Note that the first three equalities here come directly from the proof of Proposition 2.2. Then $\sigma_1^p=\sigma_3^{a_1}\sigma_4^{b_1}$. Furthermore, it can be quickly verified that $\sigma_4\sigma_1=\sigma_1\sigma_4$, $\sigma_3\sigma_1=\sigma_1\sigma_3\sigma_4^{b_3}$, and $\sigma_2\sigma_1=\sigma_1\sigma_2\sigma_3^{a_0}\sigma_4^{b_0}$. Then using these facts and those we know about $\Gal(K_4/K_1)\subset G$ from Proposition 2.2, we have
	  \begin{alignat*}{1}
	  	G=G_{\vec{a},\vec{b}}=\langle \sigma_1&,\sigma_2,\sigma_3,\sigma_4:\sigma_1^p=\sigma_3^{a_1}\sigma_4^{b_1},\sigma_2^p=\sigma_4^{b_2},\sigma_3^p=\sigma_4^p=1,\sigma_4\in Z(G_{\vec{a},\vec{b}}),\\ &\sigma_3\sigma_2=\sigma_2\sigma_3,
	  	\sigma_3\sigma_1=\sigma_1\sigma_3\sigma_4^{b_3},\sigma_2\sigma_1=\sigma_1\sigma_2\sigma_3^{a_0}\sigma_4^{b_0}\rangle,
	  \end{alignat*}
	  exactly as presented in equation (4) above. This completes the proof of the converse.
	  	  
\end{proof}

We will now consider degree $p^4$ extensions with the final Galois group. To do so, we will first need a lemma similar to Lemma 2.1.

\begin{lemma}
	Let $K$ be a field of characteristic $p$ and $L/K$ a Galois extension of degree $p^2$ with Galois group $G=\gen{\sigma_1}$. Let $x_1,x_3\in L$ be such that $L=K(x_1,x_3)$, $(\sigma_1-1)x_1=1$, $(\sigma_1-1)x_3=C_1(x_1)$, $\wp(x_1)=\beta_1\in K$, and $\wp(x_3)=D_1(x_1)+\beta_3$ for some $\beta_3\in K$, with $C_1,D_1\in \bF_p[x_1]$ as defined in Lemma 2.1. Let $C_2$ be the polynomial $$C_2(x_1,x_3)=\frac{x_1^{p^2}+1-(x_1+1)^{p^2}+p(x_3^p-(x_3+C_1(x_1))^p)}{p^2}\in\bF[x_1,x_3].$$ Then $\Tr_{L/K}(C_2(x_1,x_3))=(1+\sigma_1+\dots+\sigma_1^{p^2-1})C_2(x_1,x_3)=1$ and $\wp(C_2(x_1,x_3))=(\sigma_1-1)D_2(x_1,x_3)$, where $$D_2(x_1,x_3)=\frac{x_1^{p^2}+\beta_1^{p^2}-(x_1+\beta_1)^{p^2}+p\left(x_3^p+\beta_3^p-(x_3+D_1(x_1)+\beta_3)^p\right)}{p^2}\in\bF_p[x_1,x_3].$$
\end{lemma}

\begin{proof}
	This proof will follow much like that of Lemma 2.1. As before, some steps will actually require using polynomial identities in polynomial rings over $\bQ$. For details on how such methods work, see Lemma 2.1; here, a simplified version that slightly abuses notation will be presented to avoid overly complicated expressions.
	
	First, note that as before, $\sigma_1^k(x_1)=x_1+k$ for every $k\in\bN$. Furthermore, since $\sigma_1$ is an automorphism of $L$ which fixes $K$ and $\sigma_1(x_3)=x_3+C_1(x_1)$, it is straightforward to see that $\sigma_1^k(x_3)=x_3+\sum_{i=0}^{k-1}C_1(x_1+i)$ for $k\in\bN$. Then
	\begin{align*}
		\Tr_{L/K}(C_2(x_1,x_3))&=\sum_{i=0}^{p^2-1}\frac{(x_1+i)^{p^2}+1-(x_1+i+1)^{p^2}}{p^2}\\&\quad+\sum_{i=0}^{p^2-1}\frac{p((x_3+\sum_{j=0}^{i-1}C_1(x_1+j))^p-(x_3+\sum_{j=0}^{i}C_1(x_1+j))^p)}{p^2}\\
		&=1+\frac{x_1^{p^2}-(x_1+p^2)^{p^2}}{p^2}+\frac{x_3^p-(x_3+\sum_{j=0}^{p^2-1}C_1(x_1+j))^p}{p}
	\end{align*}
	Note that $x_1^{p^2}-(x_1+p^2)^{p^2-1}=-\sum_{i=0}^{p^2}\binom{p^2-1}{i}x_1^ip^{2(p^2-i)}$. Note that aside from when $i=p^2-1$, each of these summands has at least a factor of $p^3$. Since $\binom{p^2}{p^2-1}=p^2$, then in fact every summand has a factor of at least $p^3$. Then since $L$ is a field of characteristic $p$, the first fraction in the above expression is equal to 0. For the second fraction, it will help to use the definition of $C_1$:
	\begin{align*}
		\sum_{j=0}^{p^2-1}C_1(x_1+j)&=\sum_{j=0}^{p^2-1}\frac{(x_1+j)^p+1-(x_1+j+1)^p}{p}\\&=\frac{x_1^p+p^2-(x_1+p^2)^p}{p}\\&=p-\sum_{i=0}^{p-1}\frac{\binom{p}{i}}{p}x_1^ip^{2(p-i)}\\
		&=pz.
	\end{align*}
	Here, we simply note that the expression is divisible by $p$, then denote by $z$ the result after factoring out $p$. Then the second fraction in the above expression for $\Tr_{L/K}(x_1,x_3)$ becomes $$\frac{x_3^p-(x_3+pz)^p}{p}=\frac{x_3^p-x_3^p-\sum_{i=0}^{p-1}\binom{p}{i}x_3^i(pz)^{p-i}}{p}\equiv 0\modulo{p}.$$ Then $\Tr_{L/K}(C_1(x_1,x_3))=1$.
	
	We must now show that $\wp(C_2(x_1,x_3))=(\sigma_1-1)D_2(x_1,x_3)$. To do so, it will help to keep in mind that $\wp$ is an $\bF_p$ linear map. Since $C_2(x_1,x_3)\in\bF_p[x_1,x_3]$, we know that $\wp(C_2(x_1,x_3))=C_2(x_1^p,x_3^p)-C_2(x_1,x_3)$. Then:
	\begin{align*}
		\wp(C_2(x_1,x_3))&=-C_2(x_1,x_3)+\frac{(x_1+\beta_1)^{p^2}+1-(x_1+\beta_1+1)^{p^2}}{p^2}\\&\quad+\frac{p((x_3+D_1(x_1)+\beta_3)^p-(x_3+C_1(x_1)+D_1(x_1+1)+\beta_3)^p)}{p^2}\\
		&=\frac{(x_1+\beta_1)^{p^2}-x_1^{p^2}-(x_1+\beta_1+1)^{p^2}+(x_1+1)^{p^2}}{p^2}\\
		&\quad+\frac{p((x_3+D_1(x_1)+\beta_3)^p-x_3^p)}{p^2}\\&\quad+\frac{p(-(x_3+C_1(x_1)+D_1(x_1+1)+\beta_3)^p+(x_3+C_1(x_1))^p)}{p^2}
	\end{align*}
	Finally, one will note that $(\sigma_1-1)D(x_1,x_3)=D(x_1+1,x_3+C_1(x_1))-D(x_1,x_3)$. Inspection will show that this is exactly the expression obtained here for $\wp(C_2(x_1,x_3))$.
\end{proof}

\begin{theorem}
	Let $K_4/K_0$ be a Galois extension of degree $p^4$ with Galois group of the form
	\begin{equation*}
		H=\langle \sigma_1,\sigma_2: \sigma_1^{p^3}=\sigma_2^p=1, \sigma_2\sigma_1=\sigma_1^{1+p^2}\sigma_2\rangle,
	\end{equation*}
	and all notation as above. Assume that $K_3/K_0$, $x_1$, $x_2$, and $x_4$ are as in Proposition 2.2. Then we can select $x_4$ such that $K_4=K_3(x_4)$, and $$x_4^p-x_4=\beta_2 x_1+D_2(x_1,x_3)+\beta_4,$$ with $\beta_4\in K_0$ and $D_2$ is the Witt polynomial from Theorem 1.1.
	
	Moreover, the converse holds; that is, if $K_4/K_0$ is an algebraic extension of degree $p^4$ with $K_4=K_0(x_1,x_2,x_3,x_4)$ and $\wp(x_i)$ as above for $i=1,2,3,4$, then $K_4/K_0$ is a Galois extension with Galois group $H$.
\end{theorem}

\begin{proof}
	As before, define $\sigma_3=\sigma_1^p$ and $\sigma_4=\sigma_1^{p^2}$. Then note that $H/\langle\sigma_4\rangle\cong C_{p^2}\times{C_p}$. Then by Proposition 2.2, we can choose $x_3\in K_3$ such that $\wp(x_3)=D_1(x_1)+\beta_3$, with $(\sigma_2-1)x_3=0$ and $(\sigma_1-1)x_3=C_1(x_1)$ (where $C_1$ and $D_1$ are as defined in Lemma 2.1). Define $J_2=K_1(x_3)$ and note that $J_2$ is the fixed field of $\langle \sigma_1^{p^2},\sigma_2\rangle$. Then since $\langle \sigma_1^{p^2},\sigma_2\rangle\cong C_p^2$, we can choose $x_4\in K_4$ such that $(\sigma_2-1)x_4=0$.
	
	Now all that remains is to determine the action of $\sigma_1$ on $x_4$. To that end, let $A=(\sigma_1-1)x_4\in K_4$. Since $\sigma_4\in Z(H)$, $(\sigma_4-1)A=(\sigma_1-1)(\sigma_4-1)x_4=(\sigma_1-1)1=0$. Then $A\in K_3$. Now using a known identity from the group definition, we derive:
	\begin{alignat*}{1}
		x_4+\sigma_2(A)&=\sigma_2(x_4+A)\\&=\sigma_2\sigma_1(x_4)\\
		&=\sigma_1^{1+p^2}\sigma_2(x_4)\\
		&=\sigma_1^{1+p^2}(x_4)\\
		&=\sigma_1(x_4+1)\\
		&=x_4+A+1
	\end{alignat*} 
	Then $(\sigma_2-1)A=1$. Therefore, $(\sigma_2-1)(A-x_2)=0$, so $A-x_2\in J_2$.
	
	Now $\Tr_{J_2/K_0}(A-x_2)=(\sigma_1^{p^2}-1)x_4=(\sigma_4-1)x_4=1$. Then by the lemma, $\Tr_{J_2/K_0}(A-x_2-C_2(x_1,x_3)))=0$. Using the additive version of Hilbert's Theorem 90, we then know that there is some $j_2\in J_2$ such that $(\sigma_1-1)j_2= A-x_2-C_2(x_1,x_3)$. Then $(\sigma_1-1)(x_4-j_2)=x_2+C_2(x_1,x_3)$. Without loss of generality, we can now replace $x_4$ with $x_4-j_2$, which maintains all prior relationships and establishes the identity $(\sigma_1-1)x_4=x_2+C_2(x_1,x_3)$. This also gives us that $(\sigma_3-1)x_4=(\sigma_1^p-1)x_4=\Tr_{K_1/K_0}A=C_1(x_3)$.
	
	Now consider $B=\wp(x_4)\in K_3$. Since $\wp$ commutes with the elements of $H$, $(\sigma_2-1)B=\wp((\sigma_2-1)x_4)=0$. Then $B\in J_2$. Now note that 
	\begin{alignat*}{1}
		(\sigma_1-1)B&=\wp(x_2+C_2(x_1,x_3))\\
		&=\beta_2+(\sigma_1-1)D_2(x_1,x_3)\\
		&=(\sigma_1-1)(\beta_2 x_1+D_2(x_1,x_3)),
	\end{alignat*}
	so $(\sigma_1-1)(B-A_1-D_2(x_1,x_3))=0$, with $A_1=\beta_2 x_1$ as before. Therefore, $B-A_1-D_2(x_1,x_3)=\beta_4$ for some $\beta_4\in K_0$. Then $B=A_1+D_2(x_1,x_3)+\beta_4$.
	
	The converse holds as well, though this part will not be shown in detail. However, it follows in much the same way as in the proof of Theorem 3.1.
	
\end{proof}

As a final note, one will note that the result stated in Theorem 1.1 simply follows by combining Proposition 2.2 and Theorems 3.1 and 3.3.

\begin{acknowledgments}
	This work was partially supported by a NASA Nebraska Space Grant Fellowship. I would also like to thank my undergraduate advisor, Dr. Griff Elder, for his mentorship, support, and guidance in this project.
\end{acknowledgments}

\newpage
\bibliographystyle{amsalpha}

\bibliography{bib}

\providecommand{\bysame}{\leavevmode\hbox to3em{\hrulefill}\thinspace}
\providecommand{\MR}{\relax\ifhmode\unskip\space\fi MR }
\providecommand{\MRhref}[2]{%
  \href{http://www.ams.org/mathscinet-getitem?mr=#1}{#2}
}
\providecommand{\href}[2]{#2}
\begin{thebibliography}{{Eld}23}

\bibitem[BCT22]{bct}
Jakub Byszewski, Gunther Cornelissen, and Djurre Tijsma, \emph{Automata and
  finite order elements in the {N}ottingham group}, Journal of Algebra
  \textbf{602} (2022), 484--554.

\bibitem[Bur55]{burnside}
W.~Burnside, \emph{Theory of groups of finite order}, Dover Publications, Inc.,
  New York, 1955, 2d ed.

\bibitem[Cam97]{camina}
Rachel Camina, \emph{Subgroups of the {Nottingham} group}, Journal of Algebra
  \textbf{196} (1997), no.~1, 101--113.

\bibitem[EK23a]{elder1}
G.~Griffith {Elder} and Kevin {Keating}, \emph{{A converse to the {Hasse-Arf}
  theorem}}, arXiv e-prints (2023), arXiv:2302.00222.

\bibitem[EK23b]{elder3}
\bysame, \emph{{Galois scaffolds for $p$-extensions in characteristic $p$}},
  arXiv e-prints (2023), arXiv:2308.02775.

\bibitem[{Eld}23]{elder2}
G.~Griffith {Elder}, \emph{{Upper ramification sequences of nonabelian
  extensions of degree $p^3$ in characteristic $p$}}, arXiv e-prints (2023),
  arXiv:2303.01984.

\bibitem[Jam80]{james}
Rodney James, \emph{The groups of order p 6 (p and odd prime)}, Mathematics of
  Computation \textbf{34} (1980), 613.

\bibitem[Sal78]{saltman}
David~J. Saltman, \emph{Noncrossed product p-algebras and {Galois}
  p-extensions}, Journal of Algebra \textbf{52} (1978), no.~2, 302–314.

\bibitem[Wea18]{weaver}
Bradley Weaver, \emph{The local lifting problem for ${D}_4$}, Israel Journal of
  Mathematics \textbf{228} (2018), 587--626.

\bibitem[Wit36]{witt}
Ernst Witt, \emph{Konstruktion von galoisschen {Körpern der Charakteristik p
  zu vorgegebener Gruppe der Ordnung} pf.}, Journal für die reine und
  angewandte Mathematik (Crelles Journal) \textbf{1936} (1936), no.~174,
  237–245.

\end{thebibliography}
\end{document}